\documentclass{amsart}
\usepackage[utf8]{inputenc}
\usepackage[hidelinks]{hyperref}
\usepackage[inline]{enumitem}
\usepackage{amsthm, amsmath, amssymb, colonequals, comment, stmaryrd, tabularx, xcolor}

\newtheorem{corollary}{Corollary}
\newtheorem{lemma}{Lemma}
\newtheorem{proposition}{Proposition}
\newtheorem{theorem}{Theorem}

\theoremstyle{definition}
\newtheorem{definition}{Definition}
\newtheorem{example}{Example}

\theoremstyle{remark}
\newtheorem{remark}{Remark}

\DeclareMathOperator{\id}{id}
\DeclareMathOperator{\Supp}{Supp}

\begin{document}

\title{Simplicity of non-associative skew Laurent polynomial rings}
\author{Per B\"ack}
\address[Per B\"ack]{Division of Mathematics and Physics, M\"alar\-dalen  University,  Box  883,  SE-721  23  V\"aster\r{a}s, Sweden}
\email[corresponding author]{per.back@mdu.se}

\author{Johan Richter}
\address[Johan Richter]{Department of Mathematics and Natural Sciences, Blekinge Institute of Technology, SE-371 79 Karlskrona, Sweden}
\email{johan.richter@bth.se}


\begin{abstract}
We introduce non-associative skew Laurent polynomial rings and characterize when they are simple. Thereby, we generalize results by Jordan, Voskoglou, and Nystedt and Öinert.
\end{abstract}

\maketitle

\section{Introduction}
In 1903, Hilbert \cite{Hil03} introduced a ring of formal Laurent series with a \emph{skewed} or \emph{twisted} multiplication to show the existence of a non-commutative ordered division ring. Nowadays, the rings and their corresponding multiplication are thus referred to as \emph{skew} or \emph{twisted Laurent series rings} and \emph{Hilbert's twist} \cite{Lam01}, respectively. Thirty years later, Ore~\cite{Ore33} initiated the study of what he called ``non-commutative polynomial rings'', today more commonly known as \emph{Ore extensions}. Since their introductions, skew Laurent series rings, Ore extensions, and the closely related \emph{skew Laurent polynomial rings} have been studied quite extensively (see e.g. \cite{GW04, Lam01, MR01} for comprehensive introductions). Moreover, some years ago, Nystedt, {\"O}inert, and Richter \cite{NOR18} introduced a non-associative generalization of Ore extensions.

In this article, we introduce a non-associative generalization of skew Laurent polynomial rings and characterize when such rings are simple. Thereby, we extend results on simplicity of skew Laurent polynomial rings by Jordan \cite{Jor84} (see \autoref{thm:jordan-simplicity}) and Voskoglou \cite{Vos87} (see \autoref{thm:voskoglou-several-simplicity}) to the non-associative setting (\autoref{thm:simplicity} and \autoref{thm:several-simplicity}, respectively). Moreover, our construction of \emph{non-associative skew Laurent polynomial rings} is a generalization of that by Nystedt and Öinert \cite{NO20}, and we obtain a generalization (\autoref{thm:several-simplicity}) of a simplicity result by them (see \autoref{thm:nystedt-oinert-several-simplicity}).

The article is organized as follows:

In \autoref{sec:prel}, we provide conventions and preliminaries from non-associative ring theory (\autoref{subsec:non-assoc-ring-theory}). We also recall some results about graded non-associative rings (\autoref{subsec:graded-ring-theory}) and remind what skew Laurent polynomial rings are\\ (\autoref{subsec:skew-laurent-ore}).

In \autoref{sec:non-assoc-laurent}, we introduce non-associative skew Laurent polynomial rings and examples thereof.

In \autoref{sec:simplicity}, we characterize when non-associative skew Laurent polynomial rings are simple. We then apply our results to the examples introduced in \autoref{sec:non-assoc-laurent}.

\section{Preliminaries}\label{sec:prel}
\subsection{Non-associative ring theory}\label{subsec:non-assoc-ring-theory}
We denote by $\mathbb{N}$ the natural numbers, including zero. By a \emph{non-associative ring}, we mean a unital ring which is not necessarily associative. If $R$ is a non-associative ring, by a \emph{left $R$-module}, we mean an additive group $M$ equipped with a biadditive map $R\times M\to M$, $(r,m)\mapsto rm$ for any $r\in R$ and $m\in M$. A subset $B$ of $M$ is a basis if for any $m\in M$, there are unique $r_b\in R$ for $b\in B$, such that $r_b=0$ for all but finitely many $b\in B$, and $m=\sum_{b\in B}r_bb$. A left $R$-module that has a basis is called \emph{free}.

For a non-associative ring $R$, the \emph{commutator} is the function $[\cdot,\cdot]\colon R\times R\to R$ defined by $[r,s]=rs-sr$ for any $r,s\in R$. The \emph{commuter} of $R$, denoted by $C(R)$, is the additive subgroup $\{r\in R\colon [r,s]=0 \text{ for all } s\in R\}$ of $R$. The \emph{associator} is the function $(\cdot,\cdot,\cdot)\colon R\times R\times R\to R$ defined by $(r,s,t)=(rs)t-r(st)$ for all $r,s,t\in R$. Using the associator we define three sets: the \emph{left nucleus} of $R$, $N_l(R)\colonequals\{r\in R\colon (r,s,t)=0 \text{ for all } s,t\in R\}$, the \emph{middle nucleus} of $R$, $N_m(R)\colonequals\{s\in R\colon (r,s,t)=0 \text{ for all } r,t\in R\}$, and the \emph{right nucleus} of $R$, $N_r(R)\colonequals\{t\in R\colon (r,s,t)=0 \text{ for all } r,s\in R\}$. From the so-called associator identity 
\begin{equation*}
u(r,s,t)+(u,r,s)t+(u,rs,t)=(ur,s,t)+(u,r,st)
\end{equation*}
which holds for all $r,s,t,u\in R$, it follows that $N_l(R)$, $N_m(R)$, and $N_r(R)$ are all associative subrings of $R$. We also define the \emph{nucleus} of $R$, $N(R)\colonequals N_l(R)\cap N_m(R)\cap N_r(R)$, and the \emph{center} of $R$, $Z(R)\colonequals C(R)\cap N(R)$. 

The next two propositions are standard results in non-associative ring theory (see e.g. the proofs of \cite[Proposition 2.1 and Proposition 2.3]{NOR18}).
\begin{proposition}\label{prop:center}
Let $R$ be a non-associative ring. Then the following equalities hold:
\begin{enumerate}[label=(\roman*)]	
\item $Z(R)=C(R)\cap N_l(R)\cap N_m(R)$;\label{it:commuter-left-middle} 
\item $Z(R)=C(R)\cap N_l(R)\cap N_r(R)$;\label{it:commuter-left-right}
\item $Z(R)=C(R)\cap N_m(R)\cap N_r(R)$.\label{it:commuter-middle-right} 
\end{enumerate}
\end{proposition}

\begin{proposition}\label{prop:center-is-field}
If $R$ is a simple non-associative ring, then $Z(R)$ is a field.    
\end{proposition}

Let $R$ be a non-associative ring. Take $u\in R$. Recall that $u$ is said to be left (right) invertible if there is $v\in R$ ($w\in R$) such that $vu=1$ ($uw=1$); in that case $v$ (or $w$) is called a left (or right) inverse of $u$. We let $R^\times$ denote the set of elements of $R$ that are both left and right invertible.

\begin{remark}\label{rem:unique-inverse}
Suppose $u\in N_m(R)\cap R^\times$. It is easy to show that $u$ has a unique left inverse $v$, $u$ has a unique right inverse $w$, and $v=w$. We let $u^{-1}$ denote the element $v=w$.
\end{remark}

The following small result should be known. However, we have not been able to find a reference, and so we provide a proof of it.

\begin{lemma}\label{lem:nuclei}
Let $R$ be a non-associative ring and let $u\in N_m(R)\cap R^\times$. Then the following assertions hold:
\begin{enumerate}[label=(\roman*)]
	\item If $u\in N_l(R)$, then $u^{-1}\in N_l(R)$;\label{it:lem:nuclei1}
	\item If $u\in N_r(R)$, then $u^{-1}\in N_r(R)$;\label{it:lem:nuclei2}
    \item If $u\in N_r(R)$ and $\left(u,u^{-1},R\right)=\{0\}$, then $u^{-1}\in N_m(R)\cap N_r(R)$.\label{it:lem:nuclei3}
\end{enumerate}
\end{lemma}

\begin{proof}
Let $r,s\in R$ and $u\in N_m(R)\cap R^\times$. By \autoref{rem:unique-inverse}, $u$ has a unique two-sided inverse $u^{-1}$, so the statement makes sense. Now we have the following:  

\ref{it:lem:nuclei1}: Let $u\in N_l(R)$. Then
\begin{align*}
u^{-1}(rs)&=u^{-1}\left(\left(\left(uu^{-1}\right)r\right)s\right)=u^{-1}\left(\left(u\left(u^{-1}r\right)\right)s\right)=u^{-1}\left(u\left(\left(u^{-1}r\right)s\right)\right)\\
&=\left(u^{-1}u\right)\left(\left(u^{-1}r\right)s\right)=\left(u^{-1}r\right)s.
\end{align*}

\ref{it:lem:nuclei2}: Let $u\in N_r(R)$. Then
\begin{align*}
r\left(su^{-1}\right)&=\left(r\left(su^{-1}\right)\right)\left(uu^{-1}\right)=\left(\left(r\left(su^{-1}\right)\right)u\right)u^{-1}=\left(r\left(\left(su^{-1}\right)u\right)\right)u^{-1}\\
&=\left(r\left(s\left(u^{-1}u\right)\right)\right)u^{-1}=(rs)u^{-1}.
\end{align*}

\ref{it:lem:nuclei3}: Let $u\in N_r(R)$ and $(u,u^{-1},R)=\{0\}$. By \ref{it:lem:nuclei2}, $u^{-1}\in N_r(R)$, and
\begin{align*}
r\left(u^{-1}s\right)&=\left(r\left(u^{-1}u\right)\right)\left(u^{-1}s\right)=\left(\left(ru^{-1}\right)u\right)\left(u^{-1}s\right)=\left(ru^{-1}\right)\left(u\left(u^{-1}s\right)\right)\\
&=\left(ru^{-1}\right)\left(\left(uu^{-1}\right)s\right)=\left(ru^{-1}\right)s.\qedhere
\end{align*}
\end{proof}
\noindent For a general introduction to non-associative algebra, see e.g. Schafer's book \cite{Sch66}.

\subsection{Graded non-associative rings}\label{subsec:graded-ring-theory}
In \cite{NO20}, Nystedt and {\"O}inert study group-graded non-associative rings. Recall that a non-associative ring $R$ is said to be \emph{graded by a group $G$}, or \emph{$G$-graded}, if there is a collection of additive subgroups $\{R_g\}_{g\in G}$ of $R$, called \emph{homogenous components}, such that $R=\bigoplus_{g\in G} R_g$ and $R_gR_h \subseteq R_{gh}$ hold for all $g,h\in G$. An ideal $I$ of $R$ is called \emph{graded} if $I=\bigoplus_{g\in G} I\cap R_g$. The ring $R$ is said to be \emph{graded simple} if the only graded ideals of $R$ are $\{0\}$ and $R$. The group $G$ is called \emph{hypercentral} if every non-trivial factor group of $G$ has a non-trivial center. In particular, all abelian groups are hypercentral.

With the terminology introduced above, the authors then prove the following theorem: 
\begin{theorem}[{\cite[Theorem 4]{NO20}}]\label{thm:nystedt-oinert-hypercentral}
    If a non-associative ring is graded by a hypercentral group, then the ring is simple if and only if it is graded simple and the center of the ring is a field. 
\end{theorem}

If $R$ is a $G$-graded non-associative ring, then we define $\Supp(R)\colonequals\{g\in G\colon R_g\neq\{0\}\}$. The ring $R$ is said to be \emph{faithfully $G$-graded} if for any $g,h\in\Supp(R)$ and non-zero $r\in R_g$, we have $rR_h\neq \{0\}\neq R_hr$. Denoting the identity element of $G$ by $e$, the authors then use \autoref{thm:nystedt-oinert-hypercentral} to prove the following result:  
\begin{corollary}[{\cite[Corollary 32]{NO20}}]\label{cor:NO}
    If $R$ is a faithfully $G$-graded ring with $\Supp(R)=G$, where $G$ is a torsion-free hypercentral group, then $R$ is simple if and only if $R$ is graded simple and $Z(R)\subseteq R_e$ holds. 
\end{corollary}

Let $\sigma_1,\ldots,\sigma_n$ be automorphisms of $R$. We say that an ideal $I$ of $R$ is a \emph{$(\sigma_1,\ldots,\sigma_n)$-ideal} if $\sigma_i(I)=I$ holds for any $i\in\{1,\ldots,n\}$. Moreover, $R$ is said to be \emph{$(\sigma_1,\ldots,\sigma_n)$-simple} if $\{0\}$ and $I$ are the only $(\sigma_1,\ldots,\sigma_n)$-ideals of $R$. Note that the definition of $(\sigma_1, \ldots, \sigma_n)$-simplicity in \cite{NO20} contains a mistake and should be the same as the one just given, which has also been confirmed by its authors. In loc.~cit., non-associative skew Laurent polynomial rings are then introduced as a class of \emph{non-associative skew group rings}. With the corrected definition of $(\sigma_1, \ldots, \sigma_n)$-simplicity above, the authors then prove the following theorem:

\begin{theorem}[{\cite[Theorem 52]{NO20}}]\label{thm:nystedt-oinert-several-simplicity}
   Let $R$ be a non-associative ring with pairwise commuting automorphisms $\sigma_1,\ldots,\sigma_n$. Then $R[X_1^\pm,\ldots,X_n^\pm;\sigma_1,\ldots,\sigma_n]$ is simple if and only if $R$ is $(\sigma_1,\ldots,\sigma_n)$-simple and there do not exist $u\in N(R)^\times$ and a non-zero $(m_1,\ldots,m_n)\in\mathbb{Z}^n$, such that for all $r\in R$ and $i\in\{1,\ldots,n\}$, the following equalities hold:
\begin{enumerate}[label=(\roman*)]
\item $(\sigma_1^{m_1}\circ \cdots \circ \sigma_n^{m_n})(r)=u^{-1}ru$;
\item $\sigma_i(u)=u$.
\end{enumerate}	
\end{theorem}

In \autoref{sec:non-assoc-laurent}, we will generalize the above construction of non-associative skew Laurent polynomial rings, and in \autoref{sec:simplicity}, prove that a generalization of \autoref{thm:nystedt-oinert-several-simplicity} holds for them.

\subsection{Skew Laurent polynomial rings}\label{subsec:skew-laurent-ore}
Let us recall the definition of (associative) skew Laurent polynomial rings.

\begin{definition}[Skew Laurent polynomial ring]
Let $S$ be a ring, $R$ a subring of $S$ containing the multiplicative identity element $1$, and $x\in S^\times$. Then $S$ is called a \emph{skew Laurent polynomial ring} of $R$ if the following axioms hold:
\begin{enumerate}[label=(S\arabic*)]
\item $S$ is a free left $R$-module with basis $\{1,x,x^{-1},x^2,x^{-2},\ldots\}$;\label{it:s1}	
\item $xR=Rx$;\label{it:s2}
\item $S$ is associative.\label{it:s3}
\end{enumerate}
\end{definition}
To construct skew Laurent polynomial rings, one considers \emph{generalized Laurent polynomial rings} $R[X^\pm;\sigma]$ where $R$ is an associative ring and $\sigma\colon R\to R$ is an automorphism. As an additive group, $R[X^\pm;\sigma]$ equals the ordinary Laurent polynomial ring $R[X^\pm]$. The multiplication in $R[X^\pm;\sigma]$ is then defined by the biadditive extension of the relations
\begin{equation}
\left(rX^m\right)\left(sX^n\right)=(r\sigma^m(s))X^{m+n}\label{eq:laurent-mult}
\end{equation}	
for any $r,s\in R$ and $m,n\in\mathbb{Z}$. In particular, the product \eqref{eq:laurent-mult} makes $R[X^\pm;\sigma]$ a $\mathbb{Z}$-graded ring and a Laurent polynomial ring of $R$ with $x=X$ (see e.g. the proof of \autoref{prop:laurent-generalized}). Moreover, every skew Laurent polynomial ring of $R$ is isomorphic to a generalized Laurent polynomial ring $R[X^\pm;\sigma]$ (see e.g. the proof of \autoref{prop:generalized-laurent}).

The next theorem, due to Jordan \cite{Jor84}, characterizes when generalized Laurent polynomial rings are simple:

\begin{theorem}[{\cite[Theorem O]{Jor84}}]\label{thm:jordan-simplicity}
Let $R$ be an associative ring with a ring automorphism $\sigma$. Then $R[X^\pm;\sigma]$ is simple if and only if $R$ is $\sigma$-simple and there do not exist $u\in R^\times$ and a non-zero $n\in\mathbb{Z}$, such that for all $r\in R$, the following equalities hold:
\begin{enumerate}[label=(\roman*)]
\item $\sigma^n(r)=u^{-1}ru$;
\item $\sigma(u)=u$.
\end{enumerate}	
\end{theorem}

If $\sigma_1,\ldots,\sigma_n$ are pairwise commuting automorphisms, then we may construct an iterated generalized Laurent polynomial ring of $R$ as follows (see also Exercise 1W in \cite{GW04}). First, we set $S_1\colonequals R[X_1^\pm;\sigma_1]$. Then $\sigma_2$ extends to a ring automorphism $\widehat{\sigma}_2$ on $S_1$, defined by $\widehat{\sigma}_2(rX_1^m)=\sigma_2(r)X_1^m$ for any $m\in\mathbb{Z}$. Next, we set $S_2\colonequals S_1[X_2^\pm;\widehat{\sigma}_2]$. Once $S_i$ has been constructed for some $i<n$, we define $S_{i+1}\colonequals S_i[X_{i+1}^\pm;\widehat{\sigma}_{i+1}]$ where $\widehat{\sigma}_{i+1}$ is the ring automorphism on $S_i$ defined by $\widehat{\sigma}_{i+1}(rX_1^{m_1}\cdots X_n^{m_n})=\sigma_{i+1}(r)X_1^{m_1}\cdots X_n^{m_n}$. We may now construct an iterated generalized Laurent polynomial ring $R[X_1^\pm;\sigma_1]\cdots[X_n^\pm;\widehat{\sigma}_n]$, which we denote by $R[X_1^\pm,\ldots,X_n^\pm;\sigma_1,\ldots,\sigma_n]$. 

Voskoglou \cite{Vos87} has generalized \autoref{thm:jordan-simplicity} and characterized when iterated generalized Laurent polynomial rings are simple:

\begin{theorem}[{\cite[Corollary 3.7]{Vos87}}]\label{thm:voskoglou-several-simplicity}
   Let $R$ be an associative ring with pairwise commuting automorphisms $\sigma_1,\ldots,\sigma_n$. Then $R[X_1^\pm,\ldots,X_n^\pm;\sigma_1,\ldots,\sigma_n]$ is simple if and only if $R$ is $(\sigma_1,\ldots,\sigma_n)$-simple and there do not exist $u\in R^\times$ and a non-zero $(m_1,\ldots,m_n)\in\mathbb{Z}^n$, such that for all $r\in R$ and $i\in\{1,\ldots, n\}$, the following equalities hold:
\begin{enumerate}[label=(\roman*)]
\item $(\sigma_1^{m_1}\circ \cdots \circ \sigma_n^{m_n})(r)=u^{-1}ru$;
\item $\sigma_i(u)=u$.
\end{enumerate}	
\end{theorem}
We note that \autoref{thm:voskoglou-several-simplicity} is the special case of \autoref{thm:nystedt-oinert-several-simplicity} when $R$ is associative.

\section{Non-associative skew Laurent polynomial rings}\label{sec:non-assoc-laurent}
We wish to define non-associative skew Laurent polynomial rings in an analogous fashion to how non-associative Ore extensions in \cite{NOR18} are defined; hence we follow the same line of reasoning as in loc.~cit. First, we note that the product \eqref{eq:laurent-mult} equips the additive group $R[X^\pm;\sigma]$ of generalized Laurent polynomials over any non-associative ring $R$ with a non-associative ring structure for any additive bijection $\sigma$ on $R$ respecting $1$. In order to define non-associative skew Laurent polynomial rings, we therefore wish to adopt the axioms \ref{it:s1}, \ref{it:s2}, and \ref{it:s3} so that these rings still correspond to the above generalized Laurent polynomial rings. We suggest the following definition:

\begin{definition}[Non-associative skew Laurent polynomial ring]
Let $S$ be a non-associative ring, $R$ a subring of $S$ containing the multiplicative identity element $1$, and $x\in S^\times$. Then $S$ is called a \emph{non-associative skew Laurent polynomial ring} of $R$ if the following axioms hold:
\begin{enumerate}[label=(N\arabic*)]
\item $S$ is a free left $R$-module with basis $\{1,x,x^{-1},x^2,x^{-2},\ldots\}$;\label{it:n1}	
\item $xR=Rx$;\label{it:n2}
\item $(S,S,x)=(S,x,S)=\{0\}$.\label{it:n3}
\end{enumerate}
\end{definition}

Note that $x^{-1}$ in the above definition does indeed exist by \ref{it:n3} and \autoref{rem:unique-inverse}. Moreover, \ref{it:n3} together with \ref{it:lem:nuclei2} in \autoref{lem:nuclei} ensure that the elements $x$ and $x^{-1}$ are power associative, so that $x^m$ is well defined for any $m\in\mathbb{Z}$. 

Let $R$ be a non-associative ring. We denote by $R[X^\pm;]$ the set of formal sums $\sum_{i\in \mathbb{Z}}r_iX^i$ where $r_i\in R$ is zero for all but finitely many $i\in \mathbb{Z}$, equipped with pointwise addition. Now, let $\sigma$ be an additive bijection on $R$ respecting $1$. The \emph{generalized Laurent polynomial ring $R[X^\pm;\sigma]$} over $R$ is defined as the additive group $R[X^\pm]$ with multiplication defined by \eqref{eq:laurent-mult}. One readily verifies that this makes $R[X^\pm;\sigma]$ a $\mathbb{Z}$-graded non-associative ring.  

\begin{remark}\label{rem:laurent-ore-monoid}
If $R$ is a non-associative ring with an additive bijection $\sigma$ that respects $1$, then $R[X^\pm;\sigma]$ is a so-called \emph{Ore monoid ring} $R[G;\pi]$ as introduced in \cite{NOR19}. Here $G=\mathbb{Z}$ and $\pi=\{\pi_b^a\}_{a,b\in G}$ where $\pi_b^a=\sigma^a$ if $a=b$ and $\pi_b^a=0$ otherwise.
\end{remark}

\begin{proposition}\label{prop:Xassoc}
Let $R$ be a non-associative ring with an additive bijection $\sigma$ that respects $1$. If $S=R[X^\pm;\sigma]$, then $X^n \in N_m(S) \cap N_r(S)$ for any $n\in\mathbb{Z}$.
\end{proposition}

\begin{proof}
This is a special case of \cite[Proposition 3]{NOR19}.
\end{proof}

\begin{proposition}\label{prop:laurent-generalized}
Let $R$ be a non-associative ring with an additive bijection $\sigma$ that respects $1$. Then $R[X^\pm;\sigma]$ is a non-associative skew Laurent polynomial ring of $R$ with $x=X$.
\end{proposition}

\begin{proof}
$R[X^\pm;\sigma]$ is clearly a non-associative ring, and $R$ can be identified with a subring of $R[X^\pm;\sigma]$ that contains the multiplicative identity element. $X$ has a two-sided inverse, and so we only need to verify that the axioms \ref{it:n1}, \ref{it:n2}, and \ref{it:n3} hold. For the proof that \ref{it:n1} holds, we refer the reader to the proof of \cite[Proposition 3.2]{NOR18} which is nearly identical. For \ref{it:n2}, we have $XR=\left(1X\right) \left(RX^0\right)=\sigma(R)X\subseteq RX$ and $RX=\sigma\left(\sigma^{-1}(R)\right)X$
 $=\left(1X\right) \left(\sigma^{-1}(R)X^0\right)=X\sigma^{-1}(R)\subseteq XR$. That \ref{it:n3} holds follows from Proposition \ref{prop:Xassoc}.\qedhere
\end{proof}

\begin{proposition}\label{prop:generalized-laurent}
Every non-associative skew Laurent polynomial ring of $R$ is isomorphic to a generalized Laurent polynomial ring $R[X^\pm;\sigma]$.
\end{proposition}

\begin{proof}
The proof is similar to the proof of \cite[Proposition 3.3]{NOR18}. However, computations are a bit more involved, and hence we provide it here for the convenience of the reader. 

Let $R$ be a non-associative ring, $S$ a skew Laurent polynomial ring of $R$ defined by $x$, and $r,s\in R$. Then, from \ref{it:n1} and \ref{it:n2}, $xr=\sigma(r)x$ for some unique coefficient $\sigma(r)\in R$. Moreover, from \ref{it:n2}, $rx=x\overline{\sigma}(r)$ for some $\overline{\sigma}(r)\in R$, and since $x\in N_m(S)$ from \ref{it:n3}, $x^{-1}(rx)=x^{-1}(x\overline{\sigma}(r))=(x^{-1}x)\overline{\sigma}(r)=1\overline{\sigma}(r)=\overline{\sigma}(r)1$. Using \ref{it:n1}, $\overline{\sigma}(r)$ is then unique, and hence $\sigma$ and $\overline{\sigma}$ define functions $\sigma\colon R\to R$ and $\overline{\sigma}\colon R\to R$. We have $rx=x\overline{\sigma}(r)=\sigma(\overline{\sigma}(r))x$, and so by \ref{it:n1}, $\sigma\circ \overline{\sigma}=\mathrm{id}_R$. We also have that $xr=\sigma(r)x=x\overline{\sigma}(\sigma(r))$. Using that $x\in N_m(S)$, $r=(x^{-1}x)r=x^{-1}(xr)=x^{-1}(x\overline{\sigma}(\sigma(r)))=(x^{-1}x)\overline{\sigma}(\sigma(r))=\overline{\sigma}(\sigma(r))$. Hence $\overline{\sigma}\circ\sigma=\mathrm{id}_R$, and so we can conclude that $\sigma$ is bijective with $\sigma^{-1}=\overline{\sigma}$. Since $rx=x\sigma^{-1}(r)$, we have $(x^{-1}r)x=x^{-1}(rx)=x^{-1}(x\sigma^{-1}(r))=(x^{-1}x)\sigma^{-1}(r)=\sigma^{-1}(r)$ and therefore $x^{-1}r=(x^{-1}r)(xx^{-1})=((x^{-1}r)x)x^{-1}=\sigma^{-1}(r)x^{-1}$. Now, on the one hand $x(r+s)=\sigma(r+s)x$, and on the other hand, $x(r+s)=xr+xs=\sigma(r)x+\sigma(s)x$ by distributivity. Hence, by \ref{it:n1}, $\sigma$ is additive. Since the multiplicative identity element in $R$ is the multiplicative identity element $1$ in $S$ by assumption, $1x=x1=\sigma(1)x$, and so by \ref{it:n1}, $\sigma(1)=1$. 

We claim that $\left(rx^m\right)\left(sx^n\right)=\left(r\sigma^m(s)\right)x^{m+n}$ for any $m,n\in\mathbb{Z}$. To prove this, as an intermediate step, let us first show that $x(x^{-1}u)=u$ for any $u\in S$. Indeed, by \ref{it:n1}, we may set  $u=\sum_{i\in\mathbb{Z}}u_ix^i$ for some $u_i\in R$. Then $x(x^{-1}u)=x\left(x^{-1}\left(\sum_{i\in\mathbb{Z}}u_ix^i\right)\right)=x\left(\sum_{i\in\mathbb{Z}}x^{-1}(u_ix^i)\right)=x\left(\sum_{i\in\mathbb{Z}}(x^{-1}u_i)x^i\right)$\\
\noindent$=x\left(\sum_{i\in\mathbb{Z}}(\sigma^{-1}(u_i)x^{-1})x^i\right)=x\left(\sum_{i\in\mathbb{Z}}\sigma^{-1}(u_i)(x^{-1}x^i)\right)=x\left(\sum_{i\in\mathbb{Z}}\sigma^{-1}(u_i)x^{i-1}\right)=\sum_{i\in\mathbb{Z}}x(\sigma^{-1}(u_i)x^{i-1})=\sum_{i\in\mathbb{Z}}(x\sigma^{-1}(u_i))x^{i-1}=\sum_{i\in\mathbb{Z}}(u_ix)x^{i-1}=\sum_{i\in\mathbb{Z}}u_ix^{i}=u$. Since $x\in N_m(S)\cap N_r(S)$ by assumption, from \ref{it:lem:nuclei3} in \autoref{lem:nuclei} it now follows that $x^{-1}\in N_m(S)\cap N_r(S)$. Recall from \autoref{subsec:non-assoc-ring-theory} that $N_m(S)$ and $N_r(S)$ are rings, and hence they are closed under multiplication. Therefore $x^n\in N_m(S)\cap N_r(S)$ for any $n\in\mathbb{Z}$. 

Now, let us return to the proof of the equality $\left(rx^m\right)\left(sx^n\right)=\left(r\sigma^m(s)\right)x^{m+n}$ for any $m,n\in\mathbb{Z}$. First, we prove by induction on $m$ that $x^ms=\sigma^m(s)x^m$. The base case $m=0$ is immediate. We now split the induction step into two cases. First, assume that $m\leq0$ and set $p=-m$. We have $x^{-(p+1)}s=x^{m-1}s=(x^mx^{-1})s=x^m(x^{-1}s)=x^m(\sigma^{-1}(s)x^{-1})=(x^m\sigma^{-1}(s))x^{-1}=(\sigma^{m}(\sigma^{-1}(s))x^{m})x^{-1}\\=\sigma^{m-1}(s)x^{m-1}=\sigma^{-(p+1)}(s)x^{-(p+1)}$, which completes the negative part of the induction step. Now assume that $m\geq0$. Then $x^{m+1}s=(x^mx)s=x^m(xs)=x^m(\sigma(s)x)=(x^m\sigma(s))x=(\sigma^m(\sigma(s))x^m)x=\sigma^{m+1}(s)x^{m+1}$, which completes the positive part of the induction step. Since $x^m,x^n\in N_m(S)\cap N_r(S)$, we get $(rx^m)(sx^n)=r(x^m(sx^n))=r((x^ms)x^n)=r((\sigma^m(s)x^m)x^n)=r(\sigma^m(s)(x^mx^n))=r(\sigma^m(s)x^{m+n})=(r\sigma^m(s))x^{m+n}$. 

Last, define a function $f\colon S\to R[X^\pm;\sigma]$ by the additive extension of the relations $f(rx^m)=rX^m$ for any $r\in R$ and $m\in\mathbb{Z}$. Then $f$ is an isomorphism of additive groups, and moreover, for any $r,s\in R$ and $m,n\in\mathbb{Z}$, $f\left(\left(rx^m\right)\left(sx^n\right)\right)=f\left(\left(r\sigma^m(s)\right)x^{m+n}\right)=\left(r\sigma^m(s)\right)X^{m+n}=\left(rX^m\right)\left(sX^n\right)=f\left(rx^m\right)f\left(sx^n\right)$.
\end{proof}

\begin{proposition}\label{prop:S-associativity}
Let $R$ be a non-associative ring with an additive bijection $\sigma$ that respects $1$. If $S=R[X^\pm;\sigma]$, then
\begin{enumerate}[label=(\roman*)]
	\item $R\subseteq N_l(S)$ if and only if $R$ is associative;\label{it:R-associative}
	\item $X\in N_l(S)$ if and only if $\sigma$ is an automorphism;\label{it:x-left-nucleus}
	\item $S$ is associative if and only if $R$ is associative and $\sigma$ is an automorphism.\label{it:S-associative}
\end{enumerate}
\end{proposition}

\begin{proof}
\ref{it:R-associative}: This follows from \cite[Proposition 7]{NOR19}.

\ref{it:x-left-nucleus}: By \cite[Proposition 7]{NOR19}, we know that $(X,S,S)\subseteq (X,R,R)S$. So it is enough to prove that $(X,R,R)= \{ 0\}$ if and only if $\sigma$ is an automorphism. However, if $r,s \in R$, then the condition $X(rs)=(Xr)s$ is clearly equivalent to $\sigma(rs) = \sigma(r)\sigma(s)$.

\ref{it:S-associative}: The conditions are clearly necessary. Conversely, if they are satisfied then also $X^{-1} \in N(S)$, so $S$ is generated by elements that belong to $N(S)$ and must be associative. 
\end{proof}

\begin{example}\label{ex:complex-skew-laurent}
	On the ring $\mathbb{C}$ we can define $\sigma_q(a+bi)= a+qbi$ for any $a,b\in \mathbb{R}$ and $q\in\mathbb{R}^\times$. Then $\sigma_q$ is an additive bijection that respects $1$, and  we can accordingly define $\mathbb{C}[X^{\pm};\sigma_q]$. Moreover, $\sigma_q$ is a ring automorphism if and only if $q=\pm 1$, and so by \ref{it:S-associative} in \autoref{prop:S-associativity}, $\mathbb{C}[X^{\pm};\sigma_q]$ is associative if and only if $q=\pm 1$.
\end{example}

\begin{example}\label{ex:quantum-torus}   
Let $T$ be a non-associative ring and let $q\in Z(T)^\times$. Define a ring automorphism $\sigma_q\colon T[Y^\pm]\to T[Y^\pm]$ by the $T$-algebra extension of the relation $\sigma_q(Y)=qY$. The \emph{non-associative quantum torus} over $T$ is the generalized Laurent polynomial ring $T[Y^\pm][X^\pm;\sigma_q]$. By \ref{it:S-associative} in \autoref{prop:S-associativity}, $T[Y^\pm][X^\pm;\sigma_q]$ is associative if and only if $T$ is associative.
\end{example}

\begin{example}\label{ex:formal-series}
Let $f\colon\mathbb{N}\to\mathbb{N}$ be any bijection with $f(0)=0$. Suppose $K$ is a field and put $R=K[[Y]]$. Let $\sigma$ be the additive bijection on $R$ defined by the $K$-linear and continuous (in the usual topology of formal power series rings) extension of the relations $\sigma(Y^n)=Y^{f(n)}$, if $n\neq1$, and $\sigma(Y)=Y^{f(1)}+1$. Note that $\sigma$ respects 1 and is well defined since $\lim_{n \to \infty} f(n) = \infty$. However, $\sigma$ is not a ring homomorphism since $\sigma(Y^2)=Y^{f(2)}$ but $\sigma(Y)^2=(Y^{f(1)}+1)^2=Y^{2f(1)}+2Y^{f(1)}+1\neq Y^{f(2)}$. By \ref{it:S-associative} in \autoref{prop:S-associativity}, $R[X^\pm;\sigma]$ is not associative. Later (see \ref{it:formal-series} in \autoref{cor:examples-simplicity}) we will use a specific choice of bijection denoted by $g$ and defined by
\begin{equation*}
g(n)=\begin{cases}0\phantom{^k}\text{ if } n=0,\\
	2^k\text{ if } n=2^{k+1}\text{ for some } k\in\mathbb{N},\\
	\min\{j\in\mathbb{N}\colon j>n \text{ and } j \text{ is not a power of } 2\}\text{ otherwise}.
\end{cases}
\end{equation*}
Hence $g(1)=3$, $g(2)=1$, $g(3)=5$, and so on. We denote the corresponding additive bijection by $\sigma_g$.
\end{example} 

Recall that a \emph{ring anti-automorphism} $\sigma$ on a non-associative ring $R$ is an additive bijection such that for any $r,s\in R$, $\sigma(rs)=\sigma(s)\sigma(r)$. In particular, $\sigma(1)\sigma(r)=\sigma(r1)=\sigma(r)=\sigma(1r)=\sigma(r)\sigma(1)$, so by the uniqueness of $1$, $\sigma(1)=1$. Hence any ring anti-automorphism $\sigma$ on $R$ naturally gives rise to a generalized Laurent polynomial ring $R[X^\pm;\sigma]$.

\begin{lemma}\label{lem:anti-automorphism}
Let $R$ be a non-associative ring with a ring anti-automorphism $\sigma$. Then $R[X^\pm;\sigma]$ is associative if and only if $R$ is associative and commutative.
\end{lemma}

\begin{proof}
By \ref{it:S-associative} in \autoref{prop:S-associativity}, $R[X^\pm;\sigma]$ is associative if and only if $R$ is associative and $\sigma$ is a ring automorphism. We claim that $\sigma$ is a ring automorphism if and only if $R$ is commutative. It is clear that $\sigma$ is a ring automorphism if $R$ is commutative. To prove the converse, assume that $\sigma$ is a ring automorphism. For any $r,s\in R$, $r=\sigma(r')$ and $s=\sigma(s')$ for some $r',s'\in R$. Since $\sigma$ is both a ring automorphism and a ring anti-automorphism, $rs=\sigma(r')\sigma(s')=\sigma(r's')=\sigma(s')\sigma(r')=sr$.  
\end{proof}

\begin{example}\label{ex:matrix-permutation}
Let $R$ be a non-associative ring and let $\sigma$ be any of the $n!(n^2-n)!$ maps on the non-associative matrix ring $M_n(R)$ defined by permuting diagonal and non-diagonal elements separately. Then $\sigma$ is an additive bijection that respects $1$, and so we can define $M_n(R)[X^\pm;\sigma]$. As a concrete example, one could e.g. take $\sigma$ to be matrix transpose, $\sigma_T$. Since $\sigma_T$ is an anti-automorphism, by \autoref{lem:anti-automorphism}, $M_n(R)[X^\pm;\sigma_T]$ is associative if and only if $n=1$ and $R$ is commutative, or if $R$ is the zero ring.
\end{example}

Let $K$ be an associative and commutative ring, and let $A$ be a non-associative $K$-algebra. Recall that an \emph{involution} of $A$ is a $K$-linear map $*\colon A\to A$, also written $a\mapsto a^*$, such that $(ab)^*=b^*a^*$ and $(a^*)^*=a$ hold for any $a,b\in A$. In particular, $*$ is an anti-automorphism. A non-associative algebra with an involution $*$ is referred to as a non-associative \emph{$*$-algebra} (for an introduction to $*$-algebras, see e.g. \cite[Section 2.2]{McC04}). If $A$ is a non-associative $*$-algebra, then $A$ naturally gives rise to a generalized Laurent polynomial ring $A[X^\pm;*]$, which by \autoref{lem:anti-automorphism} is associative if and only if $A$ is associative and commutative. 

\begin{example}\label{ex:matrix-conjugate-transpose}
$M_n(\mathbb{C})$ with $*$ given by conjugate transpose is an associative $*$-algebra over $\mathbb{R}$ which is commutative if and only if $n=1$. Hence, by \autoref{lem:anti-automorphism}, $M_n(\mathbb{C})[X^\pm;*]$ is associative if and only if $n=1$ (in which case we get the ring $\mathbb{C}[X^\pm;\sigma_q]$ with $q=-1$ in \autoref{ex:complex-skew-laurent}).
\end{example}

Starting from any non-associative $*$-algebra, the so-called Cayley--Dickson construction gives a new  non-associative $*$-algebra. In particular, by starting from the real numbers viewed as a real $*$-algebra with $*=\id_\mathbb{R}$ and then repeatedly applying the Cayley--Dickson construction, we get the following real, non-associative $*$-algebras where $*$ is given by conjugation: the complex numbers, the quaternions ($\mathbb{H}$), the octonions ($\mathbb{O}$), and so on. For more details on this construction, see e.g. \cite{Bae02}. 

\begin{example}\label{ex:cayley-staralgebra}
Let $A$ be any of the real, non-associative $*$-algebras $\mathbb{R}, \mathbb{C}, \mathbb{H}, \ldots$ with $*$ given by conjugation. Then $A$ is commutative if and only if $A=\mathbb{R}$ or $\mathbb{C}$. By \autoref{lem:anti-automorphism}, $A[X^\pm;*]$ is associative if and only if $A=\mathbb{R}$ or $\mathbb{C}$. 	
\end{example}

\section{Simplicity}\label{sec:simplicity}
In this section, we examine when generalized Laurent polynomial rings are simple. If $R$ is a non-associative ring with additive bijections $\sigma_1,\ldots,\sigma_n$ that respect $1$, then, just as in the case of automorphisms in \autoref{subsec:graded-ring-theory}, an ideal $I$ of $R$ is said to be a \emph{$(\sigma_1,\ldots,\sigma_n)$-ideal} if $\sigma_i(I)=I$ holds for all $i\in\{1,\ldots,n\}$. The ring $R$ is called \emph{$\sigma$-simple} if $\{0\}$ and $I$ are the only $(\sigma_1,\ldots,\sigma_n)$-ideals of $R$. Also recall from \autoref{subsec:graded-ring-theory} that an ideal $I$ of a $G$-graded ring $R=\bigoplus_{g\in G} R_g$ is called graded if $I=\bigoplus_{g\in G} I\cap R_g$, and that $R$ is said to be graded simple if the only graded ideals of $R$ are $\{0\}$ and $R$. 

With $R[X^\pm;\sigma]$ viewed as a $\mathbb{Z}$-graded ring, we have the following result:

\begin{proposition}\label{prop:gradedsimplicity}
    Let $R$ be a non-associative ring with an additive bijection $\sigma$ that respects $1$. Then $R[X^\pm;\sigma]$ is graded simple if and only if $R$ is $\sigma$-simple.
\end{proposition}

\begin{proof}
    If $I$ is a non-trival $\sigma$-ideal of $R$, then the elements in $S\colonequals R[X^\pm;\sigma]$ with coefficients from $I$ form a proper, non-zero graded ideal of $S$. 

    Conversely, suppose that $R$ is $\sigma$-simple. If $I$ is a non-zero graded ideal of $S$, then it has non-zero intersection with some homogeneous component. The coefficients of that intersection is a non-zero $\sigma$-ideal of $R$, so $I$ contains $X^n$ for some $n\in\mathbb{Z}$. However, then $I$ contains $1$, and we have $I=S$. 
\end{proof}

\begin{proposition}\label{prop:center-laurent}
Let $R$ be a non-associative ring with an additive bijection $\sigma$ that respects $1$. Then $Z(R[X^\pm;\sigma])$ equals the set of elements of the form $\sum_{i\in\mathbb{Z}}r_iX^i$, $r_i\in N_m(R)\cap N_r(R)$, such that for all $r\in R$ and $i\in\mathbb{Z}$, the following equalities hold:
\begin{enumerate}[label=(\roman*)]
\item $r_i\sigma^i(r)=rr_i$;
\item $\sigma(rr_i)=\sigma(r)r_i$.
\end{enumerate}
\end{proposition}

\begin{proof}
Let $S=R[X^\pm;\sigma]$ and $p=\sum_{i\in\mathbb{Z}}r_iX^i\in Z(S)$. For any $r\in R$,
\begin{align*}
0&=\left[\textstyle\sum_{i\in\mathbb{Z}}r_iX^i,r\right]=\textstyle\sum_{i\in\mathbb{Z}}\left[r_iX^i,r\right]=\sum_{i\in\mathbb{Z}}\left(r_iX^i\right)r-r\left(r_iX^i\right)\\
&=\textstyle\sum_{i\in\mathbb{Z}}\left(r_i\sigma^i(r)-rr_i\right)X^i,
\end{align*}
so by comparing coefficients, 
\begin{equation}
	r_i\sigma^i(r)=rr_i,\quad\text{for any } i\in\mathbb{Z}.\label{eq:pi-commute}
\end{equation}
Moreover,
\begin{align*}
0&=\left[\textstyle\sum_{i\in\mathbb{Z}}r_iX^i,X\right]=\textstyle\sum_{i\in\mathbb{Z}}\left[r_iX^i,X\right]=\textstyle\sum_{i\in\mathbb{Z}}\left(r_iX^{i+1}-\sigma(r_i)X^{i+1}\right)\\
&=\textstyle\sum_{i\in\mathbb{Z}}(r_i-\sigma(r_i))X^{i+1},
\end{align*}
so $\sigma(r_i)=r_i$ for any $i\in\mathbb{Z}$. Also,
\begin{align*}
0&=\left(X,r,\textstyle\sum_{i\in\mathbb{Z}}r_iX^i\right)=\textstyle\sum_{i\in\mathbb{Z}}\left(X,r,r_iX^i\right)=\textstyle\sum_{i\in\mathbb{Z}}(Xr)\left(r_iX^i\right)-X\left(r\left(r_iX^i\right)\right)\\
&=\textstyle\sum_{i\in\mathbb{Z}}(\sigma(r)X)\left(r_iX^i\right)-\sigma(rr_i)X^{i+1}=\textstyle\sum_{i\in\mathbb{Z}}(\sigma(r)\sigma(r_i)-\sigma(rr_i))X^{i+1},
\end{align*}
so $\sigma(rr_i)=\sigma(r)\sigma(r_i)$ for any $i\in\mathbb{Z}$. Since $\sigma(r_i)=r_i$, we have
\begin{equation}
\sigma(rr_i)=\sigma(r)r_i,\quad\text{for any } i\in\mathbb{Z}.\label{eq:pi-automorphism}
\end{equation}
Last, for any $r,s\in R$,
\begin{align*}
0&=\left(r,\textstyle\sum_{i\in\mathbb{Z}}r_iX^i,s\right)=\textstyle\sum_{i\in\mathbb{Z}}\left(r,r_iX^i,s\right)=\textstyle\sum_{i\in\mathbb{Z}}\left(r\left(r_iX^i\right)\right)s-r\left(\left(r_iX^i\right)s\right)\\
&=\textstyle\sum_{i\in\mathbb{Z}}((rr_i)\sigma^i(s)-r(r_i\sigma^i(s)))X^i,\\
0&=(r,s,\textstyle\sum_{i\in\mathbb{Z}}r_iX^i)=\textstyle\sum_{i\in\mathbb{Z}}\left(r,s,r_iX^i\right)=\textstyle\sum_{i\in\mathbb{Z}}(rs)\left(r_iX^i\right)-r\left(s\left(r_iX^i\right)\right)\\
&=\textstyle\sum_{i\in\mathbb{Z}}((rs)r_i-r(sr_i))X^i,
\end{align*}
and since $\sigma$ is surjective, so is $\sigma^i$, and hence 
\begin{equation}
r_i\in N_m(R)\cap N_r(R).\label{eq:pi-nuclei}
\end{equation}
We now prove that the conditions \eqref{eq:pi-commute}--\eqref{eq:pi-nuclei} are also sufficient. First, we claim that \eqref{eq:pi-automorphism} is equivalent to
\begin{equation}
\sigma^j(rr_i)=\sigma^j(r)r_i,\quad\text{for any } i,j\in\mathbb{Z}.	\label{eq:pi-j-automorphism}
\end{equation}
By letting $j=1$, we see that \eqref{eq:pi-j-automorphism} implies \eqref{eq:pi-automorphism}. We show the opposite implication by induction on $j$. The base case $j=0$ follows from the definition. We now split the induction step into two cases. First, assume that $j\geq 0$ and that the induction hypothesis and \eqref{eq:pi-automorphism} hold. Then $\sigma^{j+1}(rr_i)=\sigma\left(\sigma^j(rr_i)\right)=\sigma\left(\sigma^j(r)r_i\right)=\sigma\left(\sigma^j(r)\right)r_i=\sigma^{j+1}(r)r_i$. Now assume that $j\leq 0$, that the induction hypothesis and \eqref{eq:pi-automorphism} hold, and set $p=-j$. From \eqref{eq:pi-automorphism}, $rr_i=\sigma^{-1}(\sigma(r)r_i)$, so $\sigma^{-1}(r)r_i=\sigma^{-1}\left(\sigma\left(\sigma^{-1}(r)\right)r_i\right)=\sigma^{-1}(rr_i)$. Hence
$\sigma^{-(p+1)}(rr_i)=\sigma^{j-1}(rr_i)=\sigma^{-1}\left(\sigma^j(rr_i)\right)=\sigma^{-1}\left(\sigma^j(r)r_i\right)=\sigma^{-1}\left(\sigma^j(r)\right)r_i=\sigma^{j-1}(r)r_i=\sigma^{-(p+1)}(r)r_i$. Moreover, by letting $r=1$ in \eqref{eq:pi-j-automorphism}, 
\begin{equation}
\sigma^j(r_i)=r_i,\quad\text{for any } i\in\mathbb{Z}.\label{eq:pi-sigma-invariant}
\end{equation}
Now assume that the conditions \eqref{eq:pi-commute}--\eqref{eq:pi-nuclei} hold. As noted above, then \eqref{eq:pi-j-automorphism} and \eqref{eq:pi-sigma-invariant} also hold. We wish to show that $\sum_{i\in\mathbb{Z}}r_iX^i\in Z(S)$. We note that it is sufficient to show that $r_iX^i\in Z(S)$ for any $i\in\mathbb{Z}$. By \ref{it:commuter-middle-right} in \autoref{prop:center},  $Z(S)=C(S)\cap N_m(S)\cap N_r(S)$, so it is sufficient to show that $\left[r_iX^i,s_jX^j\right]=\left(s_jX^j,t_kX^k,r_iX^i\right)=\left(s_jX^j,r_iX^i,t_kX^k\right)=0$ hold for any $s_j,t_k\in R$ and $i,j,k\in\mathbb{Z}$. We have
\begin{align*}
\left[r_iX^i,s_jX^j\right]&=\left(r_iX^i\right)\left(s_jX^j\right)-\left(s_jX^j\right)\left(r_iX^i\right)=\left(r_i\sigma^i(s_j)-s_j\sigma^j(r_i)\right)X^{i+j}\\
&\stackrel{\eqref{eq:pi-commute}}{=}(s_jr_i-s_jr_i)X^{i+j}=0.
\end{align*}
Moreover, for any $r,s,t\in R$ and $i,j,k\in\mathbb{Z}$, 
\begin{align}\label{eq:associator}
\left(rX^i,sX^j,tX^k\right)&=\left(\left(rX^i\right)\left(sX^j\right)\right)\left(tX^k\right)-\left(rX^i\right)\left(\left(sX^j\right)\left(tX^k\right)\right)\\
&=\left(\left(r\sigma^i(s)\right)X^{i+j}\right)\left(tX^k\right)-\left(rX^i\right)\left(\left(s\sigma^j(t)\right)X^{j+k}\right)\nonumber\\
&=\left(\left(r\sigma^i(s)\right)\sigma^{i+j}(t)-r\sigma^i\left(s\sigma^j(t)\right)\right)X^{i+j+k},\nonumber\\
\left(s_jX^j,t_kX^k,r_iX^i\right)&\stackrel{\eqref{eq:associator}}{=}\left(\left(s_j\sigma^j(t_k)\right)\sigma^{j+k}(r_i)-s_j\sigma^j\left(t_k\sigma^k(r_i)\right)\right)X^{i+j+k}\nonumber\\
&\stackrel{\eqref{eq:pi-sigma-invariant}}{=}\left(\left(s_j\sigma^j(t_k)\right)r_i-s_j\sigma^j(t_kr_i)\right)X^{i+j+k}\nonumber\\
&\stackrel{\eqref{eq:pi-j-automorphism}}{=}\left(\left(s_j\sigma^j(t_k)\right)r_i-s_j\left(\sigma^j(t_k)r_i\right)\right)X^{i+j+k}\nonumber\\
&\stackrel{\eqref{eq:pi-nuclei}}{=}\left(\left(s_j\sigma^j(t_k)\right)r_i-\left(s_j\sigma^j(t_k)\right)r_i\right)X^{i+j+k}=0,\nonumber\\
\left(s_jX^j,r_iX^i,t_kX^k\right)&\stackrel{\eqref{eq:associator}}{=}\left(\left(s_j\sigma^j(r_i)\right)\sigma^{i+j}(t_k)-s_j\sigma^j\left(r_i\sigma^i(t_k)\right)\right)X^{i+j+k}\nonumber\\
&\stackrel{\eqref{eq:pi-sigma-invariant}}{=}\left((s_jr_i)\sigma^{i+j}(t_k)-s_j\sigma^j\left(r_i\sigma^i(t_k)\right)\right)X^{i+j+k}\nonumber\\
&\stackrel{\eqref{eq:pi-commute}}{=}\left((s_jr_i)\sigma^{i+j}(t_k)-s_j\sigma^j(t_kr_i)\right)X^{i+j+k}\nonumber\\
&\stackrel{\eqref{eq:pi-j-automorphism}}{=}((s_jr_i)\sigma^{i+j}(t_k)-s_j\left(\sigma^j(t_k)r_i)\right)X^{i+j+k}\nonumber\\
&\stackrel{\eqref{eq:pi-commute}}{=}\left((s_jr_i)\sigma^{i+j}(t_k)-s_j\left(r_i\sigma^{i+j}(t_k)\right)\right)X^{i+j+k}\nonumber\\
&\stackrel{\eqref{eq:pi-nuclei}}{=}\left((s_jr_i)\sigma^{i+j}(t_k)-(s_jr_i)\sigma^{i+j}(t_k)\right)X^{i+j+k}=0.\qedhere
\end{align}
\end{proof}

Using \autoref{prop:center-laurent}, the following result is immediate:
\begin{corollary}\label{cor:center}
Let $R$ be a non-associative ring with an additive bijection $\sigma$ that respects $1$. Then $Z(R[X^\pm;\sigma])\subseteq R$ holds if and only if there do not exist a non-zero $s\in N_m(R)\cap N_r(R)$ and a non-zero $n\in\mathbb{Z}$, such that for all $r\in R$, the following equalities hold:
\begin{enumerate}[label=(\roman*)]
\item $s\sigma^n(r)=rs$;\label{it:cor:center1}
\item $\sigma(rs)=\sigma(r)s$.\label{it:cor:center2}
\end{enumerate}	
\end{corollary}

\begin{theorem}\label{thm:simplicity}
Let $R$ be a non-associative ring with an additive bijection $\sigma$ that respects $1$. Then $R[X^\pm;\sigma]$ is simple if and only if $R$ is $\sigma$-simple and there do not exist $u\in N_m(R)\cap N_r(R)\cap R^\times$ and a non-zero $n\in\mathbb{Z}$, such that for all $r\in R$, the following equalities hold:
\begin{enumerate}[label=(\roman*)]
\item $\sigma^n(r)=u^{-1}ru$;\label{it:thm:sigma-n}
\item $\sigma(ru)=\sigma(r)u$.\label{it:thm:sigma}
\end{enumerate}	
\end{theorem}

\begin{proof}
   We note that $S\colonequals R[X^\pm;\sigma]$ is a faithfully $\mathbb{Z}$-graded ring with $\Supp(S)=\mathbb{Z}$, and that $\mathbb{Z}$ is a torsion-free hypercentral group. Hence, by \autoref{cor:NO}, \autoref{cor:center}, and \autoref{prop:gradedsimplicity}, we need only show that if $R$ is $\sigma$-simple and there is an $s \in R$ satisfying the conditions in Corollary \ref{cor:center}, then $s$ belongs to $R^\times$. To show this, consider the ideal $I$ of $R$ generated by $s$. Since $s\in N_m(R)$, we have $I=RsR$. By using \ref{it:cor:center1} in \autoref{cor:center}, we see that $I=Rs$. Hence, by \ref{it:cor:center2} in \autoref{cor:center}, we have $\sigma(I)=I$. Since $R$ is $\sigma$-simple and $I$ is non-zero, it follows that $I=R$, and so there exist $t\in R$ such that $ts=1$. This implies that $s\sigma^n(t)=1$ by \ref{it:cor:center1} in \autoref{cor:center} and that $\sigma^n(t)s=1$ by \ref{it:cor:center2} in \autoref{cor:center}, so $s\in R^\times$. 
   \end{proof}

By using the above theorem, we may deduce when the examples in \autoref{sec:non-assoc-laurent} are simple, and when they are not.

\begin{corollary}\label{cor:examples-simplicity}
The following assertions hold:
\begin{enumerate}[label=(\roman*)]
	\item $\mathbb{C}[X^\pm;\sigma_q]$ in \autoref{ex:complex-skew-laurent} is simple if and only if $q\neq\pm1$;\label{it:ex1}
	\item If $T$ is simple, then the non-associative quantum torus $T[Y^\pm][X^\pm;\sigma_q]$ in \autoref{ex:quantum-torus} is simple if and only if $q$ is not a root of unity;\label{it:ex2}
    \item $K[[Y]][X^\pm;\sigma_g]$ in \autoref{ex:formal-series} is simple;\label{it:formal-series}
	\item $M_n(R)[X^\pm;\sigma]$ in \autoref{ex:matrix-permutation}, $M_n(\mathbb{C})[X^\pm;*]$ in \autoref{ex:matrix-conjugate-transpose}, and $A[X^\pm;*]$ in \autoref{ex:cayley-staralgebra} are not simple.\label{it:ex3}
\end{enumerate}	
\end{corollary}

\begin{proof}
\ref{it:ex1}: Since $\mathbb{C}$ is simple, it is also $\sigma_q$-simple. If $q=\pm 1$, then $\sigma_q^2=\id_\mathbb{C}$, so \ref{it:thm:sigma-n} and \ref{it:thm:sigma} in \autoref{thm:simplicity} hold with $u=1$. Hence $\mathbb{C}[X^\pm;\sigma_q]$ is not simple. If $q\neq \pm 1$, then $\sigma_q^n(r)\neq r=u^{-1}ur=u^{-1}ru$ for any $r\in \mathbb{C}$ and $u\in \mathbb{C}^\times$. By \autoref{thm:simplicity}, $\mathbb{C}[X^\pm;\sigma_q]$ is simple.\\

\noindent\ref{it:ex2}: The proof is an adaptation of that of \cite[Corollary 1.18]{GW04} to the non-associative setting. Let $R=T[Y^\pm]$ and $S=R[X^\pm;\sigma_q]$. If $q$ is a root of unity, say $q^n=1$ for some non-zero $n\in\mathbb{N}$, then $\sigma_q^n(Y)=q^nY=Y$ and $\sigma_q^n(Y^{-1})=q^{-n}Y^{-1}=Y^{-1}$, so $\sigma_q^n(r)=r$ for all $r\in R$. Hence \ref{it:thm:sigma-n} and \ref{it:thm:sigma} in \autoref{thm:simplicity} hold with $u=1$, so $S$ is not simple. Conversely, assume that $q$ is not a root of unity. Let $u\in N_m(R)\cap N_r(R)\cap R^\times$ and assume that $n\in\mathbb{Z}$ is non-zero. Then $\sigma_q^n(Y)=q^nY\neq Y=u^{-1}uY=u^{-1}Yu$, so \ref{it:thm:sigma-n} in \autoref{thm:simplicity} does not hold. We note that by \autoref{thm:simplicity}, $R=T[Y^\pm;\id_T]$ is not simple. We claim, however, that it is $\sigma_q$-simple. To this end, let $I$ be a non-zero $\sigma_q$-ideal of $R$. We must show that $I=R$. We observe that $I\cap T[Y]$ is non-zero, and so we can choose a non-zero $p\in I\cap T[Y]$ of minimal degree, say $p=t_mY^m+\cdots+t_0$ for some $m\in\mathbb{N}$ and $t_{m},\ldots,t_0\in T$ where $t_m\neq 0$. If we can show that $p=t_mY^m$, then we are done: $I\cap T$ is an ideal of $T$, and if $p=t_mY^m$, then $t_m=pY^{-m}\in I$, so $I\cap T$ is non-zero. Since $T$ is simple, we must have $I\cap T=T$. In particular, $1\in T=I\cap T\subseteq I$, so $I=R$.  If $m=0$, then clearly $p=t_mY^m$. Hence, assume that $m$ is positive. Since $I$ is a $\sigma_q$-ideal, we have $\sigma_q(p)\in I\cap T[Y]$. Now, $\sigma_q(p)=q^mt_mY^m+\cdots+t_0$, and so $\sigma_q(p)-q^mp$ is in $I\cap T[Y]$ and of degree at most $m-1$. By the minimality of $m$, we must have $\sigma_q(p)-q^mp=0$, from which it follows that $q^it_i=q^mt_i$, that is, $(q^{m-i}-1)t_i=0$, for any $i\in\{1,\ldots,m\}$. Since $q$ is not a root of unity, $q^{m-i}-1$ is non-zero and therefore an element of $Z(T)\backslash\{0\}$ for any $i\in\{1,\ldots, m-1\}$. By \autoref{prop:center-is-field}, $Z(T)$ is a field, so $q^{m-i}-1$ is invertible for any $i\in\{1,\ldots, m-1\}$. In particular, $(q^{m-i}-1)t_i=0$ implies $t_i=0$ for any $i\in\{1,\ldots, m-1\}$. Consequently, $p=t_mY^m$. \\

\noindent\ref{it:formal-series} Let $R=K[[Y]]$ where $K$ is a field. Then $R$ is not simple, but we claim it is $\sigma_g$-simple. To show this, let $I$ be a non-zero  ideal of $R$. It is well known that $I$ is generated by some element $Y^m$ where $m\in\mathbb{N}$. Hence $I$ contains an element $Y^k$, where $k$ is a power of $2$. However, then it follows that if $I$ is $\sigma_g$-invariant, then it must contain $Y$, and hence also $Y^3$. Then $1=Y^3+1-Y^3=\sigma_g(Y)-Y^3\in I$, so $I=R$.

Clearly $\sigma_g^n(Y) \neq Y = u^{-1}uY=u^{-1}Yu$ for any $u \in R^\times$ and non-zero $n\in\mathbb{N}$. By \autoref{thm:simplicity}, $R[X^\pm;\sigma]$ is simple.\\

\noindent\ref{it:ex3}: For $M_n(R)[X^\pm;\sigma]$ in \autoref{ex:matrix-permutation}, $\sigma$ may be viewed as a permutation on a finite set, and since any permutation on a finite set has finite order, there is some non-zero $n\in\mathbb{Z}$ such that $\sigma^n=\id_{M_n(R)}$. Moreover, for any involution $*$ on a non-associative ring $R$, we have $*^2=\id_R$. Hence we can conclude that for all three examples, \ref{it:thm:sigma-n} and \ref{it:thm:sigma} in \autoref{thm:simplicity} hold with $u=1$, so the corresponding generalized Laurent polynomial rings are not simple.
\end{proof}

Let $R$ be a non-associative ring with pairwise commuting additive bijections $\sigma_1,\ldots,\sigma_n$ respecting $1$. Then we may construct an iterated generalized Laurent polynomial ring $R[X_1^\pm;\sigma_1]\cdots [X_n^\pm;\widehat{\sigma}_n]$, denoted by $R[X_1,\ldots,X_n;\sigma_1,\ldots,\sigma_n]$, in the same way as described in \autoref{subsec:skew-laurent-ore}. We note that this is a generalization of the non-associative skew Laurent polynomial rings introduced in \cite{NO20}; the construction in loc.~cit. corresponds exactly to the case when $\sigma_1,\ldots,\sigma_n$ are automorphisms. Moreover, we can now formulate a generalization of \autoref{thm:nystedt-oinert-several-simplicity}: 

\begin{theorem}\label{thm:several-simplicity}
   Let $R$ be a non-associative ring with pairwise commuting additive bijections $\sigma_1,\ldots,\sigma_n$ respecting $1$. Then $R[X_1^\pm,\ldots,X_n^\pm;\sigma_1,\ldots,\sigma_n]$ is simple if and only if $R$ is $(\sigma_1,\ldots,\sigma_n)$-simple and there do not exist $u\in N_m(R)\cap N_r(R)\cap R^\times$ and a non-zero $(m_1,\ldots,m_n)\in\mathbb{Z}^n$, such that for all $r\in R$ and $i\in\{1,\ldots,n\}$, the following equalities hold:
\begin{enumerate}[label=(\roman*)]
\item $(\sigma_1^{m_1}\circ \cdots \circ \sigma_n^{m_n})(r)=u^{-1}ru$;
\item $\sigma_i(ru)=\sigma_i(r)u$.
\end{enumerate}	
\end{theorem}

\begin{proof}
The proof follows the proof of \autoref{cor:center} and  \autoref{thm:simplicity} closely. First, we see that the arguments in \autoref{prop:center-laurent} can easily be adapted to the current case. If  $S=R[X_1^\pm,\ldots,X_n^\pm;\sigma_1,\ldots,\sigma_n]$, then it follows immediately that $Z(S)\subseteq R$ holds if and only if there do not exist a non-zero $s\in N_m(R)\cap N_r(R)$ and a non-zero $(m_1,\ldots,m_n)\in\mathbb{Z}^n$, such that for all $r\in R$ and $i\in\{1,\ldots,n\}$, the following equalities hold:
\begin{enumerate}[label=(\roman*)]
\item $s(\sigma_1^{m_1}\circ \cdots \circ \sigma_n^{m_n})(r)=rs$;
\item $\sigma_i(rs)=\sigma_i(r)s$.
\end{enumerate}	

That graded simplicity of $S$ is equivalent to $(\sigma_1,\ldots,\sigma_n)$-simplicty of $R$ is clear. 

The proof is then finished by an argument similar to the proof of \autoref{thm:simplicity}. 
\end{proof}

\section*{Acknowledgements}
We would like to thank the anonymous referee for valuable comments on the manuscript.

\end{document}